\documentclass[11pt,a4paper]{amsart}
\normalfont

\usepackage{amssymb}
\usepackage{color}
\usepackage{tikz}
\usepackage{tikz-cd}
\usetikzlibrary{arrows,shapes}
\usetikzlibrary{decorations.pathreplacing,decorations.markings}

\usepackage[all]{xy}

\usepackage{graphicx}

\newcommand*{\hurl}  [2][www.]{\href{http://#1#2}{\nolinkurl{#2}}}
\newcommand*{\hemail}[1]{\href{mailto:#1}{\nolinkurl{#1}}}

\renewcommand{\epsilon}{\varepsilon}
\renewcommand{\setminus}{\smallsetminus}

\numberwithin{equation}{section}

\newtheorem{theorem}[equation]{Theorem}
\newtheorem{proposition}[equation]{Proposition}
\newtheorem{corollary}[equation]{Corollary}

\newtheorem{example}[equation]{Example}

\theoremstyle{definition}

\theoremstyle{remark}
\newtheorem{remark}[equation]{Remark}

\newcommand{\Q}{\mathbb Q}
\newcommand{\Z}{\mathbb Z}
\newcommand{\N}{\mathbb N}

\newcommand{\F}{\operatorname{F}}

\newcommand{\cohom}[3]{H^{{\raise1pt\hbox{$\scriptstyle#1$}}}(#2\>\!,#3)}
\newcommand{\tatecohom}[3]%
  {\widehat H^{{\raise1pt\hbox{$\scriptstyle#1$}}}(#2\>\!,#3)}

\newcommand{\Cohom}[3]%
  {H^{{\raise1pt\hbox{$\scriptstyle#1$}}}\big(#2\>\!,#3\big)}
\newcommand{\Tatecohom}[3]%
  {\widehat H^{{\raise1pt\hbox{$\scriptstyle#1$}}}\big(#2\>\!,#3\big)}

\newcommand{\homol}[3]{H_{{\lower1pt\hbox{$\scriptstyle#1$}}}(#2\>\!,#3)}
\newcommand{\homolog}[2]{H_{{\lower1pt\hbox{$\scriptstyle#1$}}}(#2)}
\newcommand{\redhomolog}[2]{\tilde{H}_{{\lower1pt\hbox{$\scriptstyle#1$}}}(#2)}

\newcommand{\colim}{\varinjlim}

\newcommand{\coker}{\operatorname{coker}}

\newcommand{\OFG}{\mathcal O_{\mathcal F}G}

\tikzset{
  on each segment/.style={
    decorate,
    decoration={
      show path construction,
      moveto code={},
      lineto code={
        \path [#1]
        (\tikzinputsegmentfirst) -- (\tikzinputsegmentlast);
      },
      curveto code={
        \path [#1] (\tikzinputsegmentfirst)
        .. controls
        (\tikzinputsegmentsupporta) and (\tikzinputsegmentsupportb)
        ..
        (\tikzinputsegmentlast);
      },
      closepath code={
        \path [#1]
        (\tikzinputsegmentfirst) -- (\tikzinputsegmentlast);
      },
    },
  },
  mid arrow/.style={postaction={decorate,decoration={
        markings,
        mark=at position .5 with {\arrow [scale=1.5]{stealth}}
      }}},
}

\DeclareMathOperator*{\tensor}{\otimes}

\title[Towards the rational homology of generalised Thompson groups]{Towards computing the rational homology and assembly maps of generalised Thompson groups}

\author{Conchita Mart\'inez-P\'erez}
\address{Conchita Mart\'inez-P\'erez, Departamento de Matem\'aticas, Universidad de Zaragoza,
50009 Zaragoza, Spain} \email{conmar@unizar.es}

\author{Brita Nucinkis}
\address{Brita E.~A.~Nucinkis, Department of Mathematics, Royal Holloway, University of London, Egham, TW20 0EX, UK.}\email{brita.nucinkis@rhul.ac.uk}

\thanks{The first named author was supported by  Gobierno de Arag\'on, European Regional Development Funds and
MTM2015-67781-P. This paper grew out of discussions with M. Varisco, who pointed out the connection to algebraic K-theory. We thank Marco for his valuable input. We also thank the referee of a previous version of this paper for carefully reading the paper and making crucial comments.}

\bigskip
\date{\today}

\begin{document}

\begin{abstract} Let $V_r(\Sigma)$ be the generalised Thompson group defined as the automorphism group of a valid, bounded, and complete Cantor algebra. We show that that for every $n \in \N$ there is a $k>n,$ such that there exists a $k$-dimensional $n$-connected simplicial complex $K$ such that $V_r(\Sigma)$ acts on $K$ with finite stabilisers.  We also determine the number of conjugacy classes of finite cyclic subgroups of a given order $m$ in Brin-Thompson groups.
We apply our computations to the rationalised Farrell-Jones assembly map in algebraic $K$-theory.

\end{abstract}

\maketitle

\section{Introduction}

There are many well-known generalisations of Thompson's group $V$ due to, for example, Higman \cite{higman}, Stein \cite{stein}, and Brin \cite{brin}, which share many of the properties of $V$. For example, they all contain infinite torsion and are of type $\F_\infty$ \cite{brown2, fluch++}. Furthermore, the Higman-Thompson groups $V_{n,r}$ and the Brin-Thompson groups $sV$ are simple \cite{higman, brin-09}. It turns out that all these are examples of automorphism groups $V_r(\Sigma)$ of valid, bounded, and complete Cantor algebras. For notation and definitions we refer to~\cite{BritaConch, confranbri}, where it was shown that all $V_r(\Sigma)$ and  centralisers of their finite subgroups are of type $\F_\infty$,  and where an explicit description of these centralisers was given.

Brown showed in \cite{brown92} that Thompson's group $V$ is rationally acyclic. Furthermore, he conjectured that $V$ is acyclic, which was recently proved in \cite{szymik}. The main theorem of \cite{brown92} states that for every Thompson-Higman group $V_{k,r}$ and every $n \in \N$ there is a $n$-dimensional $n-1$-connected simplicial complex on which $V_{k,r}$ acts with finite stabilisers. In this paper we prove the following generalisation:

\begin{theorem}\label{cont}
Let $U_r(\Sigma)$ be a complete, valid, and bounded Cantor algebra, and $V_r(\Sigma)$ its automorphism group. Then, for every $n \in \N$ there exists a $k>n$ and a $k$-dimensional $n$-connected simplicial complex $K$ such that $V_r(\Sigma)$ acts on $K$ with finite stabilisers.
\end{theorem}

The complex $K$ will be obtained by truncating the Stein-complex $X$ as was done in \cite{brown92}. We also give a long exact sequence of the homology of the complexes $K/V_r(\Sigma),$ which potentially serve as a tool to compute the rational homology of $V_r(\Sigma).$

In \cite{BritaConch} it was shown that, for any finite subgroup $Q \leq V_r(\Sigma)$, there are only finitely many conjugacy classes of subgroups isomorphic to $Q$. Using the method employed there, in Proposition \ref{conjclass} we explicitly calculate the number of conjugacy classes of cyclic subgroups of a given order in the Brin-like groups $sV_{n,r},$  where $n$ denotes the arity of the descending operations. Note that $sV= sV_{2,1}.$

In Section~\ref{FJC} we discuss an application of our results to the algebraic $K$-theory of the integral group rings of these generalised Thompson groups.
We first review the Farrell-Jones Conjecture, and then explain how our computations, combined with the results from~\cite{LRRV, BritaConch, confranbri}, imply the following theorem.
This generalises the analysis that was carried out for Thompson's group~$T$ in~\cite{geoghegan-varisco}.
All terms and notation are explained in Section~\ref{FJC}.

\begin{theorem}
\label{assemble}
Suppose the rationalised Farrell-Jones Conjecture in algebraic $K$-theory holds for~$sV_{n,r}$.
Then $K_n(\Z[sV_{n,r}])\tensor_\Z\Q$ is isomorphic to
\begin{equation}
\label{eq:FJ-sV}
\bigoplus_{1\le m}
\bigoplus_{i=1}^{\sum_{\Omega\in A_m}(n-1)^{|\Omega|}+\beta(m,n,r)}
\bigoplus_{\substack{p+q=n\\0\leq p\\-1\le q}}
H_p(Z_{sV_{n,r}}(C^i_m);\Q)\tensor_{\Q[W_G(C^i_m)]}\Theta_{C^i_m}\Bigl(K_q(\Z[C^i_m])\tensor_{\Z}\Q\Bigr),
\end{equation}
where, for each $m\geq 1$  $\{C^i_m\}$ denotes a complete set of representatives of the conjugacy classes of finite cyclic groups of order~$m$ in~$sV_{n,r}$.

If the Leopoldt-Schneider Conjecture holds for all cyclotomic fields,
then $K_n(\Z[sV_{n,r}])\tensor_\Z\Q$ contains as a direct summand the subspace of~\eqref{eq:FJ-sV} indexed by~$q\ge0$.
\end{theorem}

We get the following almost immediate Corollary for the original Brin-Thompson groups, i.efor $n=2$ and $r=1$:

\begin{corollary}\label{cor} If the rationalised Farrell-Jones Conjecture in algebraic $K$-theory holds for~$sV$,
then $K_n(\Z[sV])\tensor_\Z\Q$ is isomorphic to

\[
\bigoplus_{1\le m}
\bigoplus_{i=1}^{2\alpha(m)}
\bigoplus_{\Phi(m)}
\bigoplus_{\substack{p+q=n\\0\leq p\\-1\le q}}
H_p(Z_{sv}(C^i_m);\Q)\tensor_{\Q[W_G(C^i_m)]}\Theta_{C_m}\Bigl(K_q(\Z[C_m^i])\tensor_{\Z}\Q\Bigr)
\]

\end{corollary}

The results used in the proof of this Theorem also imply that the homology groups $H_p(Z_G(C_m^i)$ can be computed in terms of the ordinary homology $H_p(sV)$ of $sV$.


Finally, we remark that using \cite{LRRV,confranbri} one also deduces that for any valid, bounded, and complete Cantor algebra, the rationalised Whitehead group $Wh(V_r{(\Sigma)})\otimes_\Z \Q$ is an infinite dimensional $\Q$-vector space; compare Corollary~\ref{Wh}.

\section{The connectivity of the truncated Stein complex}

Throughout this section we will use the notation used in \cite{confranbri}.
Let $U_r(\Sigma)$ denote a valid, bounded, and complete Cantor algebra, and let $G=V_r(\Sigma)$ be its automorphism group.
The examples to keep in mind are those where $r=1$, and $G=V_1(\Sigma)$ is either the Higman-Thompson-group $G=V_{n,1}$ or the Brin-Thompson group $sV$. 

The group $V_r(\Sigma)$ are defined as follows (For detail the reader is referred to \cite[Section 2]{BritaConch}): Let $S=\{1,\ldots,s\}$ be a finite set of colours
and associate to each $i \in S$ an integer  $n_i>1$, called arity of the colour $i.$ For every $i\in S$ consider  the following right operations on a set $U$: 
\begin{itemize}
\item[(i)] One $n_i$-ary operation $\lambda_i\, :\,U^{n_i}\to U,$ and 
\item[(ii)] $\alpha_i:U\to U^{n_i}$ 
\end{itemize}
The maps $\alpha_i$ are called {\sl descending} operations, or expansions, and the maps $\lambda_i$ are called {\sl ascending} operations, or contractions.

Fix a finite set $X_r$ of cardinality $|X_r|=r$. We define the Cantor-algebra $U_r(\Sigma)$ on $X_r$  to free object on the set $X_r$ with respect to the previous operations to satisfying a certain set of laws $\Sigma.$ These laws can roughly be described as requiring that expanding and then contracting with the same colour, and vice versa, yields the identity, and furthermore, that any "commutativity relation" between expansions of different colours are all of the length two, i.e. involve at most two colours.

Let $B\subset U_r(\Sigma)$, $b\in B$ and $i$ a colour of arity $n_i$. The set
$$(B\setminus\{b\})\cup\{b\alpha_i^1,\ldots,b\alpha_i^{n_i}\}$$
is called a {\sl simple expansion} of $B$. Analogously, if $b_1,\ldots,b_{n_i}\subseteq B$ are pairwise distinct,
$$(B\setminus\{b_1,\ldots,b_{n_i}\})\cup\{(b_1,\ldots,b_{n_i})\lambda_i\}$$
is a {\sl simple contraction} of $B$.
A finite chain of simple expansions (contractions) is an expansion (contraction). A  subset $A\subseteq U_r(\Sigma)$ is called {\sl admissible} if it can be obtained from the set $X_r$ by finitely many expansions or contractions. If a subset $A_1$ is obtained from a subset $A$ by an expansion (simple or not), then we write $A\leq A_1$.

 $U_r(\Sigma)$ is said to be {\sl bounded} (see \cite[Definition 2.7]{BritaConch})  if for all admissible subsets $Y$ and $Z$ such that there is some admissible $A\leq Y,Z$, there is a unique  admissible subset $T$ such that $Y \leq T$ and $Z \leq T$, and whenever there is an admissible set $S$ also satisfying $Y\leq S$ and $Z\leq S$, then $T \leq S.$
It turns out \cite[Lemma 2.5]{desiconbrita2} and \cite[Theorem 2.5]{confranbri}, that for  any valid and bouded $U_r(\Sigma)$, admissible subset are bases and vice versa.


The set of bases $\mathcal{A}$ is a poset, and we consider its geometric realisation $|\mathcal{A}|,$ which is contractible if $U_r(\Sigma)$ is bounded. The Stein complex by $X$ \cite{stein}, denoted $\mathcal{S}_r(\Sigma)$ in~\cite{confranbri}, is the following subcomplex of $|{\mathcal{A}}|:$ 

 Let $B\leq A$ be bases of $U_r(\Sigma)$. We say that the expansion $B\leq A$ is {\sl elementary} if there are no repeated colours in the paths from leaves in $B$ to their descendants in $A$.  Since $\Sigma$ is complete, this condition is preserved by the relations in $\Sigma$. We denote an elementary expansion by $B\preceq A.$ 
 The vertices of the Stein complex $X$ are  given by the admissible subsets
of $U_r(\Sigma)$, and the  $k$-simplices are given by chains of  expansions $Y_0\leq\ldots\leq Y_k$, where $Y_0\preceq Y_k$ is an elementary expansion.

\begin{remark}\label{bigbasesrem}
The Stein complex is known to be contractible \cite{brown92, fluch++, confranbri}. By the same argument  as in the remark at the end of Section 4 of \cite{brown92}, nothing changes if one fixes an integer $q$ and replaces the Stein complex by  the Stein complex obtained by only considering bases of cardinality $\geq q.$
\end{remark}

\subsection*{The truncated Stein complex}  As in \cite{brown92} we will consider a truncated Stein complex.  For $1\leq p\leq q$ let $X_{p,q}$ denote the full subcomplex of $X$ generated by bases $A$ with $p \leq |A| \leq q$. Note that $\dim X_{p,q} \leq q-p.$

Recall from \cite[Lemma 3.4]{confranbri} that each basis has a unique maximal elementary descendant, denoted $\mathcal{E}(A)$, with $n_1\dotsm n_s|A|$ leaves obtained by applying all descending operations exactly once to each element of $A$.  

\begin{remark} For $V$ we have $N=2$; for $V_k$ we have $N=k$; and for $sV$ we have $N=2^s$.
\end{remark}

The following result now immediately implies Theorem \ref{cont}.

\begin{theorem}\label{p0} Suppose that $s \geq 2$ and let $U_r(\Sigma)$ be a complete, valid, and bounded Cantor algebra with associated arities $n_1\leq\ldots\leq n_s$. For all $n\geq 1$
there is an integer $p_0$, depending on $n$, such that $X_{p,q}$ is $n$-connected for $p \geq p_0$ and $$q+1-p\geq\left\{
\begin{aligned} & (n+2)(n_s-1),\\
&(n+1)(N-1)\\
\end{aligned}\right.$$ 
\end{theorem}

\begin{proof}

Since the complex given by the union of
$$X_{p,p+tn} \subset X_{p,p+tn+1} \subset \dotsb$$
is contractible, see Remark \ref{bigbasesrem}, it suffices to show that there is some $p_0$ such that for $p\geq p_0$ and $q\geq p+tn$, the pair is  $(X_{p,q+1},X_{p,q})$ is $n$-connected, i.e., the inclusion $X_{p,q}\hookrightarrow X_{p,q+1}$
induces an isomorphism  between the homotopy groups $\pi_i$ for $0\leq i\leq n$. Exactly as in  \cite[Theorem 2]{brown92}, this will  be satisfied if for any $A$ with $|A|=q+1$, then
$$L^p(A)=|\{B\in X_{p,q}\mid B<A\}|$$
is $n$-connected. We follow the lines of the argument of \cite{fluch++}, and begin by showing  that
$$L^p_0(A)=|\{B\in X_{p,q}\mid B<A\, \text{very elementary}\}|$$
is $n$-connected.  Consider the complex $K_q,$ which was denoted $K_n$ in the proof of \cite[Lemma 3.8]{confranbri}. The vertices of this complex are pairs $(i,S)$, where $i$ is a color and $S$ is an ordered subset of $A$ whose cardinality is the arity $n_i$. A $k$-simplex is a set $\{(i_0,S_0),\ldots,(i_k,S_k)\}$ for which $S_0,\ldots,S_k$ are pairwise disjoint. As in \cite[Lemma 3.8]{confranbri}, the argument of \cite[Lemma 4.20]{brown2} shows that $K_q$ is $n$-connected if $q$ is big enough. Now, we claim that the condition $q+1-p\geq (n+2)(n_s-1)$ implies that the $n+1$ skeleton is precisely the barycentric subdivision of the $n+1$-skeleton of $K_q$. To see it, note that we only have to check that for any $n+1$-simplex $\{(i_0,S_0),\ldots,(i_{n+1},S_{n+1})\}$ in $K_q$, the result of performing the associated very elementary contraction in $A$ lies in $X_{p,q}$, in other words, that has at least $p$ elements. Equivalently, we need
$$p\leq q+1-\sum_{j=0}^{n+1}(n_{i_j}-1),$$
that is
$$q+1-p\geq \sum_{j=0}^{n+1}(n_{i_j}-1),$$
which holds since
$$(n+2)(n_s-1)\geq \sum_{j=0}^{n+1}(n_{i_j}-1).$$

Thus under this condition, we deduce that  $L^p_0(A)$ is also $n$-connected.

  As in \cite{confranbri}, \cite{fluch++},  one can define a  height-function  for each $B\in L^p(A)$ as follows: For every $i=2,...,s$, let $c_i$ be the number of leaves in $B$ whose descendants in $A$ have length $i$ and put
$$h(B) = (c_s,...,c_2, |B|).$$

Now let $B\in L^p(A)\setminus L^p_0(A)$ and consider its descending link $\text{lk}^p\downarrow(B)$ in $L_p(A)$ with respect to the height function $h$, but now truncating the basis of less thatn $p$ elements $p$. We claim that this link is $(n-1)$-connected. Using Morse theory (see \cite[Lemma 3.1]{fluch++}) this will yield the result.
We have
\[\text{lk}^p\downarrow(B)=\text{downlk}^p\downarrow(B)*\text{uplk}^p\downarrow(B),\]
where
\[\text{downlk}^p\downarrow(B)=|\{C\mid C<B,\,h(C)<h(B),\,p\leq|C|\}|\]
is the truncated downlink, and
\[\text{uplk}^p\downarrow(B)=\text{uplk}\downarrow(B)=|\{C\mid B<C,\,h(C)<h(B)\}|\]
is the uplink (truncation does not affect the uplink).

 As in \cite{fluch++}, we may assume that no leaf of $A$ is obtained by a simple expansion of a leaf in $B$, otherwise
the uplink is contractible and the claim holds true. This means that each leaf of $B$ yields either some number $l$ with $n_1< l\leq N=n_1\ldots n_s$ of leaves in $A,$ or remains invariant in $A$. Let $k_{B,l}$ be the number of leaves of $B$ yielding exactly $l$ leaves in $A$ and let $k_1$ be the number of the invariant leaves of $B$.  Put $k=\sum_{l=n_1+1}^{N}k_{B,l}$. Then
\begin{equation}\label{eq1}|B|=k+k_1,\end{equation}
\noindent and
\begin{equation}\label{eq2}|A|=\sum_{l=n_1+1}^{N}lk_{B,l}+k_1=q+1.\end{equation}

Set $|B|=b$ and let $C\in  \text{downlk}^p\downarrow(B)$, i.e., let $C<B$ such that $h(C)<h(B)$. Observe that this condition implies that 
$$(c^B_s,\ldots,c^B_2)=(c^C_s,\ldots,c^C_2)$$
where the superscript indicates the relevant basis. This means that $C$ is obtained from $B$ by contracting some of the $k_1$ leaves that are left invariant in $A$ and that each leaf is contracted at most once. Therefore, $\text{downlk}^p\downarrow(B)$ is equivalent to the complex associated to
$$L^{p-k}_0(A_1)$$
where $A_1$ has exactly $k_1$ elements. 

Now, we claim that $L^{p-k}_0(A_1)$ is $(n-k-1)$-connected. Observe that once the claim is proved, taking into account that the argument of \cite{fluch++} implies that the uplink is $(k-2)$-connected we will deduce that
$\text{lk}^p\downarrow(B)$ is $(n-1)$-connected as we wanted.  

To prove the claim, using the first part of the proof we only have to check that 
$$b-p=k+k_1-p\geq(n-k-1+2)(n_s-1),$$
equivalently, that 
$$(n_s-1)k+k_1-(p-k)+n_s-1\geq(n+2)(n_s-1).$$
To see it, note that 
$$|A|=q+1=\sum_{l=n_1+1}^Nlk_{B,l}+k_1\leq N\sum_{l=n_1+1}^Nk_{B,l}+k_1=Nk+k_1$$
thus, using that $k_1-(p-k)=k+k_1-p\geq 0$,
$$\begin{aligned}
(n_s-1)k+k_1-(p-k)+n_s-1\geq\\
 {n_s-1\over N-1}((N-1)k+k_1-(p-k))+n_s-1=\\
{n_s-1\over N-1}(Nk+k_1-p)+n_s-1\geq\\ 
{n_s-1\over N-1}(q+1-p)+n_s-1\geq\\
{n_s-1\over N-1}(n+1)(N-1)+n_s-1=\\ 
(n+1)(n_s-1)+n_s-1=(n+2)(n_s-1).
\end{aligned}$$

\end{proof}

\medskip\noindent
From now on fix a basis $A$ with $|A|=p$, and
let $Y_{p,q}=X_{p,q}/G.$
Consider the complex $\widehat Z_{p+1,q}$ with $m$-simplices of the form $\sigma:B_0<B_1<\dotsb<B_m$ with $|B_m|\leq q$, $A<B_0$ and $A\leq B_m$ elementary. We let the (finite) group
$$H=\{g\in G\mid gA=A\}$$
act on $\widehat Z_{p+1,q}$,
and set
$$Z_{p+1,q}=\widehat Z_{p+1,q}/H.$$
We also denote by $CZ_{p+1,q}$ the cone on $Z_{p+1,q}.$


 \begin{proposition}\label{tec} $Y_{p+1,q}$ is a subcomplex of $Y_{p,q}$, and there is an isomorphism of chain complexes
$$\mathcal{C}_*(Y_{p,q},Y_{p+1,q})\cong\mathcal{C}_*(CZ_{p+1,q},Z_{p+1,q})$$

and therefore a long exact sequence in homology
$$\dotsb\to\homolog{j}{Y_{p+1,q}}\to\homolog{j}{Y_{p,q}}\to\redhomolog{j-1}{Z_{p+1,q}}\to\homolog{j-1}{Y_{p+1,q}}\to\dotsb.$$

  \end{proposition}

\begin{proof} Note first that $CZ_{p+1,q}$ can be seen as the complex with $m$-simplices of the form $H(B_0<B_1<\dotsb<B_m)$ with $A\leq B_0$. The fact that $Y_{p+1,q}$ lies inside $Y_{p,q}$ is obvious. Moreover, if $G\sigma$ is an orbit with
$G\sigma\in Y_{p,q}\setminus Y_{p+1,q}$, then  $\sigma$ has the form $\sigma:A_0<A_1<\dotsb<A_m$ with $|A_0|=p$. We may choose some $g\in G$ with $gA_0=A$, and map
$$\varphi:G\bar\sigma\mapsto H(\bar\nu),$$
where $\nu=A<gA_1<\dotsb<gA_n$, and $\bar{\ }$ denotes passing to the quotient chain complex.
Note that if we choose some other $g'\neq g$ with $g'A_0=A$, then $g=hg'$ for some $h\in H$, thus
$$H(A<gA_1<\dotsb<gA_n)=H(A<g'A_1<\dotsb<g'A_n).$$
In particular this map does not depend on the choice of $g$. It is obvious that $\varphi$ induces an isomorphism at each degree of the chain complexes, we only have to show that this is in fact a chain map. But this follows from the fact that, if  $\delta$ is the boundary map, then
$$\delta(G\bar\sigma)=\sum_{i=1}^m(-1)^iG(\bar\tau_i)$$
with $\tau_i=A_0<\dotsb<A_{i-1}<A_{i+1}<\dotsb<A_m.$

Now use that $\homolog{j}{CZ_{p+1,q},Z_{p+1,q}}\cong\redhomolog{j-1}{Z_{p+1,q}}.$
\end{proof}

\begin{example}\label{S1} In the case when $G=2V$, the homotopy type of the following complexes are:
$$\begin{aligned} 
Y_{q,q}\text{ is a point,}\\
Y_{q-1,q}\text{ is }S^1,\\
Y_{q-2,q}\text{ is }S^1.\\
\end{aligned}
$$
\end{example}
\begin{proof} The fact that $Y_{q,q}$ is a point is obvious. Also, $Y_{q-1,q}$ is $S^1$: it has only two points (corresponding to the classes of basis of size $q-1$ and $q$ and two edges connecting them (corresponding to expansions of the 2 possible colours), we claim that the whole $Y_{q-2,q}$ can be contracted to this subcomplex. Consider first the six 1-cells in $Y_{q-2,q}$ which are the cosets of cells in $X_{q-2,q}$ of the form $A_{q-2}<A_{q}$ associated to one of the following possible diagrams (solid lines represent colour 1, dashed represent colour 2):

\bigskip

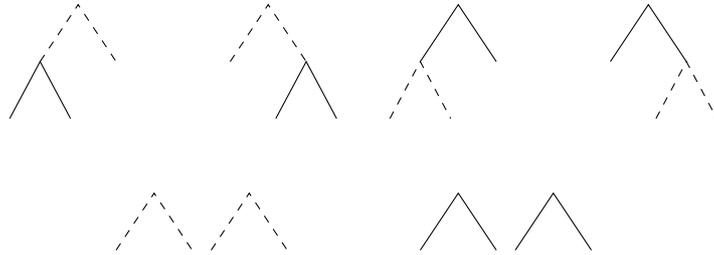
\begin{center}
\begin{tikzpicture}[scale=0.5]

  \draw[black, dashed]
    (0,0) -- (1, 1.51) -- (2,0);
 \draw[black] (-0.8,-1.51) -- (0,0) --(0.8,-1.51);
 
   \draw[black, dashed]
    (5,0) -- (6, 1.51) -- (7,0);
     \draw[black] (7-0.8,-1.51) -- (7,0) --(7.8,-1.51);

    \draw[black]
    (10,0) -- (11, 1.51) -- (12,0);
  \draw[black,dashed] (10-0.8,-1.51) -- (10,0) --(10+0.8,-1.51);
  
   \draw[black]
    (15,0) -- (16, 1.51) -- (17,0);
      \draw[black,dashed] (17-0.8,-1.51) -- (17,0) --(17.8,-1.51);


\draw[black, dashed]
    (2,-5) -- (3, 1.51-5) -- (4,-5);
\draw[black, dashed]
    (4.5,-5) -- (5.5, 1.51-5) -- (6.5,-5);

\draw[black]
    (10,-5) -- (11, 1.51-5) -- (12,-5);
\draw[black]
    (12.5,-5) -- (13.5, 1.51-5) -- (14.5,-5);

   \end{tikzpicture}
\end{center}

   \bigskip
   Each of these 1-cells in is in the boundary of a single 2-cell of $Y_{q-2,q}$.
We begin by pushing these 1-cells along the corresponding 2-cell to contract it to the other two 1-cells of its boundary. In this way we get a smaller homotopy equivalent complex having only two 2-cells $\sigma_1$ and $\sigma_2$ with 
$\sigma_1$ the class of
$A_{q-2}<_1A_{q-1}<_2 A_{q}$ where the first is a simple expansion of color 1 and the second is a simple expansion of color 2 and  $\sigma_2$ is the class 
of
$A_{q-2}<_2A_{q-1}<_1 A_{q}$ where the color are interchanged. In fact this complex can be visualized as a number eight shape with two 2-cells  sharing a common edge glued along each of the sides of the eight. The 1-skeleton of this complex is
\bigskip

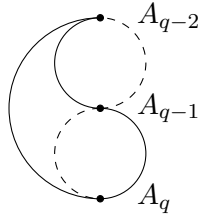
\begin{center}
\begin{tikzpicture}[scale=0.6]

 \draw[black] (0,2) arc (90:270:2);
 \draw[black] (0,2) arc (90:270:1);
 \draw[black,dashed] (0,0) arc (90:270:1);
 \draw[black,dashed] (0,0) arc (-90:90:1);
 \draw[black] (0,-2) arc (-90:90:1);

   \filldraw(0,2) circle (2pt) node[right=10pt]{$A_{q-2}$};
     \filldraw(0,0) circle (2pt) node[right=10pt]{$A_{q-1}$};

  \filldraw(0,-2) circle (2pt) node[right=10pt]{$A_{q}$};

   \end{tikzpicture}
\end{center}
   \bigskip
and the whole complex can be contracted to the copy of $S^1$ at the bottom by sending $A_{q-2}$ to $A_{q}$ along the  1-cell joining them  and pushing the 1-cells on the top to the 1-cells in the bottom.

\end{proof}

\section{Conjugacy classes of finite cyclic subgroups in~$sV_{n,r}$}

For the next result we need to introduce some notation. Let $m>0$ be an integer number. Let
$A_m$ be the set with elements 
$$\{s_1,\ldots,s_t\mid 0< s_i\neq 1,\, s_i\text{ pairwise distinct}\text{ and }\mathrm{lcm}(s_1,\ldots,s_t)=m\}$$
and let
$$\alpha(m)=|A_m|.$$
For example,  for any prime number $p$ we have $\alpha(p)=1$ and $\alpha(p^l)=2^{l-1}$. 

Now, for each $\Omega=\{s_1,\ldots,s_t\}\in A_m$, let 
$B(\Omega)$ be the set with elements of the form $\{k_{s_1},\ldots,k_{s_t}\}$ such that the $k_{s_i}$'s are integers with $0< k_{s_i}\leq n-1$ and also  
let  $B_{m,n}$ be the disjoint union of all the sets $B(\Omega)$ where $\Omega$ lies in $A_m$. Note that the cardinality of $B_{m,n}$ is
$$|B_{m,n}|=\sum_{\Omega\in A_m}(n-1)^{|\Omega|}$$
and in the particular case when $n=2$, we have
$$|B_{m,2}|=|A_m|=\alpha(m).$$
Finally, let
$$\beta(m,n,r)=|\{\{k_{s_1},\ldots,k_{s_t}\}\in B_{m,n}\mid\sum_{i=1}^tk_{s_i}s_i\equiv r\text{ mod }n-1\}|$$
and note that in the case when $n=2$ we have
$$\beta(m,2,r)=|B_{m,2}|=|A_m|=\alpha(m).$$

$\Omega$ be an element of $A_m$, i.e., $\Omega=\{s_1,\ldots s_t\}$ with the elements $s_i$ as before. 

\begin{proposition}\label{conjclass} With the previous notation, for any $m>0$   the number of conjugacy classes of cyclic groups of order $m$ in $sV_{2,r}$ is
$$2\alpha(m).$$
In the general case, the number of conjugacy classes of cyclic groups of order $m$ in $sV_{n,r}$ is
$$\sum_{\Omega\in A_m}(n-1)^{|\Omega|}+\beta(m,n,r).$$
 \end{proposition}
\begin{proof} We claim first that the number of conjugacy classes of cyclic subgroups of order~$m$ equals the cardinality of the set $C(n,r,m)$ of the possible
$$\{k_1,k_{s_1},\ldots,k_{s_t}\}$$ such that $0\leq k_1\leq n-1$, $\{k_{s_1},\ldots,k_{s_t}\}$ belongs to the set $B_m$ defined before and 
\begin{equation}\label{card}k_1+\sum_{i=i}^tk_{s_i}s_i\equiv r \mod n-1.\end{equation}
The claim follows essentially from \cite[Proposition 4.2 and Theorem 4.3]{BritaConch},  but we briefly recall the argument here.
The main idea goes back to Higman (see \cite{higman} and also  \cite{WarrenConch}): one can prove that any finite subgroup $H\leq sV_{n,r}$ acts on a certain $n$-ary $r$-rooted forest by permuting the set of leaves $L$. That set of leaves can be split into transitive subsets, each corresponding to a particular type as a permutation representation, and at least one of the orbits must be faithful. The number of orbits of a certain type can be modified modulo $n-1$ just by passing to subtrees.  This yields the same group, yet note that if some type does not appear it can not be created. Any other copy  $H_1\leq sV_{n,r}$ of the same group is conjugate to $H$ in $sV_{n,r}$ if and only if,  in
both groups for each type, the set of orbits of that precise type is either zero in both, or non zero in both and congruent modulo $n-1$.

In the particular case of a cyclic group of order $m$, the types of permutation representations correspond to the possible divisors of $m$,  more precisely, to each $s\mid m$ we may associate the transitive permutation representation of length $s$. The number $k_s$ above refers to the number of orbits of length $s.$ For this permutation representation to be faithful we need the condition that $\mathrm{lcd}(s_1,\ldots,s_t)=m$ that was used to define $A_m$.  And from the fact that our group is defined using an $r$-rooted forest we get  (\ref{card}). This proves the claim.

Next, observe that for each choice of $\{k_{s_1},\ldots,k_{s_t}\}$ in $B_{m,n}$, we must have
$$k_1\equiv r-\sum_{i=i}^tk_{s_i}s_i$$
modulo $n-1$ which
yields exactly one possible value of $k_1,$ in tha case when  $\sum_{i=i}^tk_{s_i}s_i\neq r$ modulo $n-1$ and exactly two values in the other case. Finally, the particular case when $n=2$ follows from the observations before the statement.
 
\end{proof}

More detailed formulas for the special case when $m=p^a$ for a prime $p$ can be found in \cite{higman} or \cite{WarrenConch}.

\section{Assembly maps and algebraic K-theory}\label{FJC}

In this final section we explain an application of our computations to algebraic $K$-theory.
We begin by recalling the rationalised version of the Farrell-Jones Conjecture.
For more details and background information we refer to the introduction of~\cite{LRRV} and to the references cited there.

Let $G$ be a group.
Denote by~$(\mathcal{FC})$ the set of conjugacy classes of finite cyclic subgroups of~$G$.
The centraliser of a subgroup $C \leq G$ is denoted $Z_G(C)$ and
the Weyl group $W_G(C)$ is defined as the quotient $W_G(C)=N_G(C)/Z_G(C)$ of the normaliser modulo the centraliser.
Notice that, when $C$ is finite, then $W_G(C)$ is finite, too.

For any group~$G$ and any~$n\in\Z$ there is a natural homomorphism
\begin{equation}
\label{eq:FJ}
\bigoplus_{(C)\in(\mathcal{FC})}
\bigoplus_{\substack{p+q=n\\p\ge0\\q\ge-1}}
H_p(Z_G(C);\Q)\tensor_{\Q[W_G(C)]}\Theta_C\Bigl(K_q(\Z[C])\tensor_{\Z}\Q\Bigr)
\to
K_n(\Z[G])\tensor_\Z\Q
\,,
\end{equation}
called \emph{rationalised Farrell-Jones assembly map}.
The \emph{rationalised Farrell-Jones Conjecture} asserts that \eqref{eq:FJ} is an isomorphism for every~$G$ and every~$n\in\Z$.

Here $\Theta_C$ is an idempotent endomorphism of $K_q(\Z[C])\tensor_{\Z}\Q$, whose image is a direct summand of $K_q(\Z[C])\tensor_{\Z}\Q$ isomorphic to
\begin{equation}
\label{eq:Artin-defect}
\coker\Biggl(\, \bigoplus_{D\lneqq C}
     K_q(\Z[D])\tensor_{\Z}\Q \to K_q(\Z[C])\tensor_{\Z}\Q \Biggr).
\end{equation}
The Weyl group acts via conjugation on $C$ and hence on $\Theta_C(K_q(\Z[C])\tensor_{\Z}\Q)$.
The Weyl group action on the homology comes from the fact that the space $EN_G(C)/Z_G(C)$ is a model for $BZ_G(C)$.
The dimensions of the $\Q$-vector spaces in~\eqref{eq:Artin-defect} are explicitly computed in~\cite[Theorem on page~9]{patronas} for any~$q$ and any finite cyclic group~$C$.

The following injectivity result about the Farrell-Jones Conjecture is proven in~\cite{LRRV}.

\begin{theorem}\cite[Theorem 1.13]{LRRV}
\label{LRRV}
Suppose that the following conditions are satisfied for each finite cyclic subgroup $C$ of $G$.
\begin{itemize}
\item[{[$A$]}] For every $p\geq1$, $H_p(Z_G(C);\Z)$ is a finitely
generated abelian group.
\item[{[$B$]}] For every $q\geq0$, the natural
homomorphism
\[
K_q(\Z[\zeta_c])\tensor_{\Z}\Q \to \prod_{\text{$\ell$ prime}}K_t(\Z_\ell\otimes_\Z\Z[\zeta_c],\Z_\ell)\tensor_{\Z}\Q
\]
is injective, where $c$ is the order of $C$ and $\zeta_c$ is any primitive $c$-th root of unity.
\end{itemize}
Then the restriction of the rationalised Farrell-Jones assembly map~\eqref{eq:FJ} to the summands with $q\geq 0$ is injective for each~$n$.
\end{theorem}

We remark that condition~[$B$] is conjecturally always true; more precisely, it is true if the Leopoldt-Schneider Conjecture in algebraic number theory holds for all cyclotomic fields.
For more details we refer to~\cite[Section~2]{LRRV}.

We can now prove Theorem~\ref{assemble}.



\begin{proof}[Proof of Theorem~\ref{assemble}]
For the first  statement, we only need to use in the source of the rationalised Farrell-Jones assembly map~\eqref{eq:FJ} the computation of the number of conjugacy classes of finite cyclic subgroups of Proposition \ref{conjclass}.
This establishes the first statement of the theorem.

The last statement is then a consequence of Theorem~\ref{LRRV}, because \cite[Theorems 3.1 and 4.9]{confranbri} imply that $sV_{n,r}$ and all centralisers of finite subgroups are of type $\F_\infty$, and therefore condition~[$A$] is satisfied.
\end{proof}

With a little bit of work we can now prove Corollary \ref{cor}:
\begin{proof}[Proof of Corollary \ref{cor}]
Using Proposition~\ref{conjclass} we obtain the number of conjugacy classes of finite cyclic subgroups of order $m$.
It was shown in \cite[Theorem 4.4]{BritaConch} that each centraliser is decomposed into a direct product $Z_{sV}(C_m)=Z_1\times \cdots\times Z_l$, where $l$ is the number of transitive permutation representations of $C_m$, and where each $Z_i$ fits into a short exact sequence of groups
$$0\to K_i \to Z_i \to sV\to 0$$
with $K_i$ a locally finite group. As already observed in the proof of Proposition \ref{conjclass}, $l=\phi(m).$ Now an easy spectral sequence argument shows that for all $p\geq 0$
$$H_p(Z_i;\Q) \cong H_p(sV;\Q).$$
\end{proof}

\begin{remark}
The Weyl groups $W_{sV}(C)$ for finite subgroups of~$sV$ were described in \cite[Theorem 5.1]{confranbri}. With the notation used in the proof of Proposition \ref{conjclass}, for each finite cyclic group of order $m$ there is a set of leaves $Y$ of a fixed order $|Y|=\sum_{s|m}sk_s,$  where $0\leq k_s\leq 1$ and $k_m=1.$ Note that here $n=2.$ Hence, for each representative $C_m^i$ of the conjugacy classes of cyclic subgroups of order $m$, there is a $Y_i$ with
$m \leq |Y_i|\leq \sum_{s|m}s$ and
$$W_{sV}(C_m^i)= N_{S(Y_i)}(C^i_m)/Z_{S(Y_i)}(C_m^i),$$
where $S(Y_i)$ denotes the symmetric group on $|Y_i|$ letters.
\end{remark}

Finally, we remark that \cite[Theorem 1.1]{LRRV} applies to general automorphism groups of Cantor algebras because of \cite[Theorems 3.1 and 4.9]{confranbri}, and immediately gives the following result about their Whitehead groups.

\begin{corollary}
\label{Wh}
Let $G=V_r(\Sigma)$ be the automorphism group of a valid, bounded, and complete Cantor algebra.
Then there is an injective homomorphism
$$\colim_{\OFG}Wh(H)\otimes_\Z \Q \to Wh(G)\otimes_\Z \Q,$$
where the colimit is taken over the orbit category of finite subgroups of~$G$.
In particular, $Wh(G)\otimes_\Z \Q$ is an infinite dimensional $\Q$-vector space.
\end{corollary}





\begin{thebibliography}{88}
\bibitem{Benson} D. J. Benson. Representations and cohomology II: Cohomology of groups and modules. Cambridge Un. Press, Cambridge, 1991.

\bibitem{brin} 
Matthew~G. Brin.
\newblock Higher dimensional {T}hompson groups.
\newblock {\em Geom. Dedicata}, 108:163--192, 2004.



\bibitem{brin-09} M. Brin, On the Baker�s map and the Simplicity
of the Higher Dimensional Thompson Groups $nV$, Publ. Mat. 54 (2010), no. 2, 433�439. 

\bibitem{brown2} 
Kenneth~S. Brown.
\newblock Finiteness properties of groups.
\newblock  {\em Journal of Pure and Applied Algebra}, 44 (1987), 45--75. 


\bibitem{brown92}  K. S. Brown. The geometry of finitely presented infinite simple groups.  Algorithms and classification in combinatorial group theory (Berkeley, CA, 1989), 121-136, 
Math. Sci. Res. Inst. Publ., 23, Springer, New York, 1992. 


\bibitem{Brownbook} K. S. Brown. Cohomology of groups,  Graduate Texts in Mathematics \textbf{87}. Springer-Verlag, New York, 1994.


\bibitem{WarrenConch} W. Dicks, C. Mart\'inez-P\'erez. Isomorphisms Brin-Higman-Thompson groups. Israel J. Math. 199 (2014), no. 1, 189--218.

\bibitem{fluch++}  M.G.Fluch, M. Marschler, S. Witzel and M.C.B. Zaremsky. The Brin-Thompson groups $sV$ are of type $\F_\infty$, Pacific J. Math. 266 (2013), no. 2, 283�295.


\bibitem{geoghegan-varisco}
R.~Geoghegan, and M.~Varisco
\newblock On {T}hompson's group $T$ and algebraic $K$-theory.
\newblock in Geometric and Cohomological Group Theory, LMS lecture note series 444, CUP, 34--45, (2018).
\newblock {arXiv 1401.0357}


\bibitem{higman}
G.~Higman,  {\em Finitely presented infinite simple groups,}
 Notes on Pure Mathematics 8, Australian National University, Canberra 1974
 
 \bibitem{desiconbrita2} 
D.~Kochloukova, C.~Mart\'inez-P\'erez and B.~E.~A. Nucinkis. {\em  Cohomological finiteness properties of the Brin-Thompson-Higman groups 2V and 3V,} {\em Proceedings of the Edinburgh Mathematical Society} (2) 56 no. 3, 777-804 (2013).
 
 \bibitem{LRRV} W. L\"uck,  H. Reich, J. Rognes and M. Varisco.  Algebraic K-theory of group rings and the cyclotomic trace map. Adv. Math. 304 (2017), 930--1020. 	


\bibitem{BritaConch} C.~Mart\'inez-P\'erez, and B.~E.~A. Nucinkis. Bredon cohomological finiteness conditions for 
generalisations of Thompson's groups, {\em Groups Geom. Dyn.} 7 (2013), 931--959,  

\bibitem{confranbri} C. Mart\'inez-P\'erez, F. Matucci and B. E. A. Nucinkis. Cohomological finiteness conditions and centralisers in generalisations of Thompson's group V.  Forum Math. 28 (2016), no. 5, 909--921.

\bibitem{patronas} D. Patronas,  The Artin Defect in Algebraic K-Theory, Ph.D Thesis, Freie Universit\"at Berlin (2014).

\bibitem{stein}
M.~Stein.
\newblock Groups of piecewise linear homeomorphisms.
\newblock {\em Trans. Amer. Math. Soc.}, 332(2):477--514, 1992.

\bibitem{szymik} M. Szymik, and Natalie Wahl,  The homology of the Higman-Thompson groups, Preprint (2016), arXiv:1411.5035.



\end{thebibliography}
\end{document}